\documentclass[3p,12pt]{elsarticle}
\usepackage[ruled,vlined]{algorithm2e}
\usepackage{hyperref}
\usepackage{umoline}
\usepackage{color}
\usepackage{array}
\usepackage{ctable}
\usepackage[all]{xy}

\newcommand{\reduline}[1]{}

\newcommand{\rcolored}[1]{}

\usepackage{amssymb}

 \usepackage{amsmath}
 \usepackage{amsthm}
\usepackage{multirow}
 \newtheorem{thm}{Theorem}[section]
\newtheorem{ex}[thm]{Example}
 \newtheorem{cor}[thm]{Corollary}
 \newtheorem{lem}[thm]{Lemma}
 \newtheorem{prop}[thm]{Proposition}
   
 \theoremstyle{definition}
 \newtheorem{defn}[thm]{Definition}
 \theoremstyle{remark}
 \newtheorem{rem}[thm]{Remark}

 \theoremstyle{definition}
 \newtheorem{exm}{Example}[section]
 \def\Blem{\begin{lem}}
 \def\Elem{\end{lem}}
 \newcommand{\uccomment}[1]{}
 \def\Bpr{\begin{prop}}
 \def\Epr{\end{prop}}

 \def\Bp{\begin{proof}}
 \def\Ep{\end{proof}}

 \def\Bex{\begin{exm}}
 \def\Eex{\end{exm}}

 \def\Bcor{\begin{cor}}
 \def\Ecor{\end{cor}}

 \def\Br{\begin{rem}}
 \def\Er{\end{rem}}
 \def\Bthm{\begin{thm}}
 \def\Ethm{\end{thm}}
 \def\Bd{\begin{defn}}
 \def\Ed{\end{defn}}
 \def\Beq{\begin{equation}}
 \def\Eeq{\end{equation}}

 \newcommand{\To}{\longrightarrow}

\newcommand{\tx}{\tilde{X}}
\newcommand{\ty}{\tilde{Y}}

 \def\ZZ{\mathbb{Z}}
  \def\NN{\mathbb{N}}
 \def\RR{\mathbb{R}}
 
  \def\he{\mathbb{HE}}

\journal{ }

\begin{document}

\begin{frontmatter}

\title{  On Semicovering, Subsemicovering and Subcovering Maps}
\author[]{ Majid Kowkabi }
\ead{m.kowkabi@stu.um.ac.ir}
\author[]{Behrooz Mashayekhy \corref{cor1}}
\ead{bmashf@um.ac.ir}
\author[]{Hamid Torabi}
\ead{h.torabi@ferdowsi.um.ac.ir}

\address{Department of Pure Mathematics, Center of Excellence in Analysis on Algebraic Structures, Ferdowsi University of Mashhad,
P.O.Box 1159-91775, Mashhad, Iran }
\cortext[cor1]{Corresponding author}


\begin{abstract}
In this paper, by reviewing the concept of subcovering and semicovering maps, we extend the notion of subcovering map to subsemicovering map. We present some necessary or sufficient conditions for a local homeomorphism to be a subsemicovering map. Moreover, we investigate the relationship between these conditions by some examples. Finally, we give a necessary and sufficient condition for a subsemicovering map to be semicovering.
\end{abstract}


\end{frontmatter}

\section{Introduction and Motivation}

Steinberg \cite[Section 4.2]{sting} defined a map $p: \tx \to X $ of locally path connected spaces a {\it subcovering map} if there exists a covering map $p': \ty \to X $ and a topological embedding $i: \tx \to \ty $ such that $ p' \circ i =p$. He presented a necessary and sufficient condition for a local homeomorphism $p: \tx \to X $ to be subcovering.  More precisely, he proved that a continuous map $p :\tx \to X $ of locally path connected and semilocally simply connected spaces is subcovering if and only if $p:\tx \to X$ is a local homeomorphism and any path $f$ in $\tx$ with $p \circ f$ null homotopic (in $X$) is closed, i.e, $f(0) = f(1)$ (see \cite[Theorem 4.6]{sting}).

Brazas \cite[Definition 3.1]{B1} extended the concept of covering map to semicovering map. A {\it semicovering map} is a local homeomorphism with continuous lifting of paths and homotopies. Klevdal in \cite[Definition 7]{Klevdal} simplified the notion of semicovering map as a local homeomorphism with unique path lifting and path lifting properties.

In this paper, we extend the notion of subcovering map to subsemicovering map. We call a local homeomorphism $p:\tx \to X $ a {\it subsemicovering map} if it can be extended to a semicovering map $q: \tilde{Y} \to X $, i.e, there exists a topological embedding $\varphi: \tx \rightarrow \tilde{Y} $ such that $q \circ \varphi=p$.
Moreover, if $ p_*(\pi_1(\tilde{X}, \tilde{x}_0)) =  q_*(\pi_1(\tilde{Y}, \tilde{y}_0))$, then we say that $p$ is a {\it full subsemicovering map}, when $q$ is a semicovering map, and a {\it full subcovering map} when $q$ is a covering map. Since any covering map is a semicovering map, every subcovering map is a subsemicovering map.
Note that there exists a subsemicovering map which is not a full subsemicovering map (see Example \ref{ex111}).

In Section 2, among reviewing the concept of local homeomorphism, path lifting property, unique path lifting property and semicovering map, we mention some needed results on these notions such as path homotopy theorem for local homeomorphism and lifting criterion theorem for semicovering map. Also, we recall from \cite[Lemma 3.1, 3.2]{mt2} a concept of defective lifting for local homeomorphisms which is call incomplete lifting.

In Section 3, we obtain some necessary or sufficient conditions for a local homeomorphism to be a subsemicovering map. First, by introducing a strong version of the unique path lifting property which we call it strong UPLP, we show that it is a necessary condition for a local homeomorphism to be a subsemicovering map. Also, we prove that if a local homeomorphism $p:\tx \to X $ is a subsemicovering map, then any path $f$ in $\tx$ with $p \circ f$ null homotopic (in $X$) is closed, i.e, $f(0) = f(1)$. Moreover, we show that the latter condition on a local homeomorphism $p :(\tx, \tilde{x}_0) \to (X, x_0) $ is a sufficient condition for $p$ to be subsemicovering provided that $p_* (\pi_1 (\tx, \tilde{x}_0))$ is an open subgroup of the quasitopological fundamental group $\pi_1^{qtop}(X, x_0)$ (see \cite{3} for the notion of the quasitopological fundamental group). Second, we investigate the relationship between these necessary or sufficient conditions by some examples. For instance, we show that openness of $p_* (\pi_1 (\tx, \tilde{x}_0))$ is not necessary for a local homeomorphism $p$ to be subsemicovering. Moreover, we give some examples to show that none of the two necessary conditions for a local homeomorphism to be subsemicovering are sufficient and also the sufficient condition is not necessary. Also, we show that a continuous map $p :\tx \to X $ of locally path connected spaces is full subsemicovering if and only if $p:(\tx, \tilde{x}_0) \to (X, x_0)$ is a local homeomorphism and any path $f$ in $\tx$ with $p \circ f$ null homotopic (in $X$) is closed and $p_* (\pi_1 (\tx, \tilde{x}_0))$ is an open subgroup of $\pi_1^{qtop}(X, x_0)$. Furthermore, we prove that a continuous map $p :\tx \to X $ of locally path connected spaces is full subcovering if and only if $p:(\tx, \tilde{x}_0) \to (X, x_0)$ is a local homeomorphism and any path $f$ in $\tx$ with $p \circ f$ null homotopic (in $X$) is closed and $p_* (\pi_1 (\tx, \tilde{x}_0))$ contains an open normal subgroup of $\pi_1^{qtop}(X, x_0)$. Finally, by extending the notions strong homotopy and the fundamental inverse category and monoid introduced by Steinberg \cite{sting} to semicovering maps, we give a necessary and sufficient condition for a subsemicovering map to be semicovering.

\section{Notations and Preliminaries}

In this paper, all maps $f : X \to Y$ between topological spaces
$X$ and $Y$ are continuous. We recall that a continuous map $p:\tilde{X} \To X$
is called a {\it local homeomorphism} if for every point $\tilde{x} \in \tilde{X}$,
there exists an open neighborhood $\tilde{W} $ of $\tilde{x}$ such that $p(\tilde{W})$ is open in $X$ and the restriction map  $p|_{ \tilde{W}}: \tilde{W} \To p( \tilde{W})$ is a homeomorphism. In this paper, we denote a local homeomorphism $p:\tilde{X} \To X$ by $( \tx, p)$ and assume that $\tx$ is path connected and locally path connected.

Assume that $X$ and $\tx$ are topological spaces and $p:\tilde{X} \To X$ is a continuous map.
Let $f:(Y, y_0) \to (X, x_0)$ be a continuous map and $\tilde{x}_0\in p^{-1}(x_0)$. If there exists a continuous map $\tilde{f}:(Y, y_0) \to (\tx, \tilde{x}_0)$ such that $p \circ \tilde{f} =f $, then  $\tilde{f}$ is called a {\it lifting} of $f$.
The map $p$ has {\it path lifting property} (PLP for short) if for every path $f$ in $X$, there exists a lifting $\tilde{f}:(I, 0) \to (\tx, \tilde{x}_0)$ of $f $.
Also, the map $p$ has {\it unique path lifting property} (UPLP for short)  if for every path $f$ in $X$, there is at most one lifting $\tilde{f}:(I, 0) \to (\tx, \tilde{x}_0)$ of $f $ (see \cite{s}).

Brazas \cite[Definition 3.1]{B1} generalized the concept of covering map by the phrase {\it {``A semicovering map is a local homeomorphism with continuous lifting of paths and homotopies"}}. Note that a map $p: Y \to X $ has {\it {continuous lifting of paths}} if $ \rho_p: (\rho Y )_y \to (\rho X)_{p(y)}$ defined by $\rho_p(\alpha) = p \circ \alpha $ is a homeomorphism, for all $y \in Y,$ where $(\rho Y)_y = \{\alpha: I = [0, 1] \to Y | \alpha(0) = y\}$. Also, a map $p: Y \to X $ {\it {has continuous lifting of homotopies}} if $\Phi_p: (\Phi Y )_y \to (\Phi X)_{p(y)}$
defined by $\Phi_p(\phi) = p \circ \phi$ is a homeomorphism, for all $y \in Y$, where elements of $(\Phi Y)_y$ are endpoint preserving homotopies of paths starting at $y$. He also simplified the definition of semicovering maps by showing that having continuous lifting of paths implies having continuous lifting of homotopies ( see \cite[Remark 2.5]{B2}).

The following theorem can be found in \cite[Lemma 2.1]{Klevdal} and \cite[Theorem 3.1]{mt}.
\begin{thm} \label{Homotopy}(Local Homeomorphism Homotopy Theorem for Paths).
\\Let $(\tx, p)$ be a local homeomorphism of \ $X$ with UPLP and PLP. Consider the diagram of continuous maps\\
\[
\xymatrix{
 I \ar[d]^{j} \ar[r]^{\tilde f} & (\tx, \tilde{x}_0) \ar[d]_p \\
 I\times I \ar[r]^F \ar@{-->}[ur]^{\tilde F} & (X, x_0),}
\]
where $j(t)= (t, 0)$ for all $t \in I$. Then there exists a unique continuous map $ \tilde{F}:I\times I \to \tx$ making the diagram commute.
\end{thm}

The following corollary is a consequence of the above theorem.
\begin{cor}\label{cor homotopy}
Let $p:\tilde{X}\to X $ be a local homeomorphism with UPLP and PLP. Let $x_0, x_1 \in X $ and  $f, g: I \to X$ be paths such that $f(0)=g(0)=x_0 $, $ f(1)=g(1)=x_1$ and  $\tilde{x}_0 \in p^{-1}(x_0)$. If $F:f \simeq g $ rel $\dot{I}$ and $\tilde{f}, \tilde{g}$ are the lifting of $f$ and $g$, respectively, with $\tilde{f}(0)= \tilde{x}_0= \tilde{g}(0)$, then $\tilde{F}: \tilde{f} \simeq \tilde{g}$ rel $\dot{I}$.
\end{cor}

The following theorem can be found in \cite[Corollary 2.6 and Proposition 6.2]{B1}.
\begin{thm}\label{criterion}(Lifting Criterion Theorem for Semicovering Maps).
\\If Y is connected and locally path connected, $f:(Y, y_0) \to (X, x_0)$ is continuous and $p: \tx \to X$ is a semicovering map where $\tx$ is path connected, then there exists a unique $ \tilde{f}:(Y, y_0) \to (\tx, \tilde{x}_0)$ such that $p \circ \tilde{f} = f$ if and only if $ f_*(\pi_1(Y, y_0))\subset  p_*(\pi_1(\tilde{X}, \tilde{x}_0))$.
\end{thm}

The following theorem can be concluded from \cite[Definition 7, Corollary 2.1]{Klevdal}.
\begin{thm} \label{semicovering map}
A map $p:\tilde{X} \To X$ is a semicovering map if and only if it is a local homeomorphism with UPLP and PLP.
\end{thm}

Note that there exists a local homeomorphism without UPLP and PLP and so it is not a semicovering map .
\begin{ex}\label{ex} (\cite[Example 2.4]{mt2}).
Let $\tx = ( [0, 1] \times \{0 \} ) \bigcup (\{1/2\} \times [0, 1/2) )$ with coherent topology with respect to $\{ [0, 1/2] \times \{0 \}, (1/2, 1] \times \{0 \}, \{1/2\} \times (0, 1/2)  \}$ and let $X = [0, 1]$. Define $p:\tx \to X$ by $  p(s, t) =  \begin{cases}  s & t =0 \\ s+1/2 & s = 1/2 \end{cases}. $ It is routine to check that $p$ is a local homeomorphism which does not have UPLP and PLP.
\end{ex}

In Section 3, we need the concept of incomplete lifting for local homeomorphisms which has been introduced by the authors in \cite[Lemma 3.1, 3.2]{mt2} as follows.
\begin{lem}\label{exist}  (\cite[Lemma 3.1]{mt2}).
Let $p:\tx \to X $ be a local homeomorphism, $f$ be an arbitrary path in $X$ and $ \tilde{x}_0 \in p^{-1}(f(0))$ such that there is no lifting of $f$
starting at $\tilde{x}_0$. If $ A_f = \{ t \in I \ | f|_{[0, t]} \ {\mathrm has} \mathrm{ \ a \ lifting} \  \hat{f}_t \ \mathrm{on} \  [0, t] \ \mathrm{with} \ \hat{f}_t(0)  = \tilde{x}_0\}$, then $A_f$ is open and connected. Moreover, there exists $ \alpha \in I $ such that $A_f = [0, \alpha)$.
\end{lem}
\begin{lem}\label{cts} (\cite[Lemma 3.2]{mt2}).
let $p:\tx \to X $ be a local homeomorphism with UPLP, $f$ be an arbitrary path in $X$ and $ \tilde{x}_0 \in p^{-1}(f(0))$ such that there is no lifting of $f$ starting at $\tilde{x}_0$. Then, using notation of the previous lemma, there exists a unique continuous map $\tilde{f}_{\alpha}: A_f = [0, \alpha) \to \tx$ such that $p\circ\tilde{f}_{\alpha} = f|_{[0, \alpha)}$. We call $\tilde{f}_{\alpha}$ the {\it incomplete lifting} of $f$ by $p$ starting at $\tilde{x}_0$.
\end{lem}
The following theorem is stated in \cite[Theorem 3.7]{t}.
\begin{thm}\label{cov}
For a connected, locally path connected space $X$, there is a one-to-one correspondence between its equivalent classes of connected covering spaces and the cojugacy classes of subgroups of its fundamental group $\pi_1(X, x_0)$ with open core in $\pi_1^{qtop}(X, x_0)$.
\end{thm}
The following theorem can be found in \cite[Theorem 2.21]{B2}.
\begin{thm}\label{semi}
Suppose $X$ is locally wep-connected and $x_0 \in X$. A subgroup $H \subseteq \pi_1(X, x_0)$ is open in $\pi_{1}^{qtop}(X, x_0)$ if and only if $H$ is a
semicovering subgroup of $\pi_1(X, x_0)$ .
\end{thm}

The following corollary is a consequence of the above theorem (see \cite[Corollary 3.4]{B2}).
\begin{cor}\label{c semi}
Every semicovering subgroup of $\pi_1(X, x_0)$ is open in $\pi_1^{qtop}(X, x_0)$.
\end{cor}

\section{Subsemicovering and Subcovering Maps}

Let $p:\tx \to X $ be a local homeomorphism. We are interested in finding some conditions on $p$ or $\tx$ under which the map $p$ can be extended to a semicovering map $ q: \tilde{Y} \to X $.
We recall that Steinberg \cite[Section 4.2]{sting} defined a map $p: \tx \to X $ of locally path connected and semilocally simply connected spaces a {\it subcovering map (and $\tx$ a subcover)} if there exists a covering map $p': \ty \to X $ and a topological embedding $i: \tx \to \ty $ such that $ p' \circ i =p$. We are going to extend this definition as follows:

\begin{defn} \label{subsemicovering}
Let $p:\tx \to X $ be a local homeomorphism. We say that $p$  {\it can be extended to a local homeomorphism $q: \tilde{Y} \to X $}, if there exists an embedding map $\varphi: \tx \hookrightarrow \tilde{Y} $ such that $q \circ \varphi = p$. In particular, if $q$ is a covering map, then $p$ is called a  {\it subcovering map} (see \cite[Section 4.2]{sting}) and if $q$ is a semicovering map, then we call the map $p$ a {\it subsemicovering map}. Moreover, if $ p_*(\pi_1(\tilde{X}, \tilde{x}_0)) =  q_*(\pi_1(\tilde{Y}, \tilde{y}_0))$, then we call the map $p$ {\it full subcovering} and {\it full subsemicovering}, respectively.
\end{defn}

Note that since every covering map is a semicovering map, every subcovering map is a subsemicovering map.
Also, if $p:(\tilde{X}, \tilde{x}_0) \to (X, x_0) $ can be extended to $ q: (\tilde{Y}, \tilde{y}_0)\to (X, x_0) $ via $\varphi: (\tilde{X}, \tilde{x}_0) \to  (\tilde{Y}, \tilde{y}_0)$, then $ p_*(\pi_1(\tilde{X}, \tilde{x}_0))$ is a subgroup of $q_*(\pi_1(\tilde{Y}, \tilde{y}_0))$.

The following example shows that a local homeomorphism may be extended to various covering maps.
\begin{ex}\label{ex11}
Let $X = S^1 \vee S^1 = \{e^{2 \pi it}+1 | t \in \RR \} \cup  \{e^{2 \pi it} - 1 | t \in \RR \}$ be the figure eight space, $ \tx = \RR \times \{0\}(\cup _{n \in \ZZ}\{(-1, 1)\times \{n\}\}) $ and $p: \tx \to X$ defined by $$ p(t, s)= \begin{cases}  e^{2 \pi it}+1 & s =0 \\ e^{2 \pi is}-1 & s \neq 0. \end{cases} $$
Then $p$ is a subcovering map since $p$ can be extended to the universal cover of figure eight space introduced in \cite[Section 1.3]{r} which we denote it by $ h: \tilde{Z} \to X$. Note that one can extend $p$ to the covering $ q: \ty \to X $ where $\ty = (\RR \times \{0\} ) \cup_{n \in \NN} (S^1 \times \{n\}) $ via an embedding map $\varphi : \tx \to \ty $ defined by $$\varphi(t,s) = \begin{cases}  (t,0) & s =0 \\( e^{2 \pi it}, t) & s \neq 0. \end{cases}$$
Hence $p$ can be extended to two coverings which are not equivalent since $p_*(\pi_1(\tilde{X}, \tilde{x}_0)) =h_*(\pi_1(\tilde{Z}, \tilde{z}_0)) = \{1\}$ but $\{1\}=  p_*(\pi_1(\tilde{X}, \tilde{x}_0)) \lvertneqq q_*(\pi_1(\tilde{Y}, \tilde{y}_0)) \leq \pi_1(X, x_0)$. Note that $p$ is a full subcovering map since $p_*(\pi_1(\tilde{X}, \tilde{x}_0)) =h_*(\pi_1(\tilde{Z}, \tilde{z}_0))$.
\end{ex}

The following example shows that there exists a subsemicovering map which is not a full subsemicovering map.
\begin{ex}\label{ex111}
Let $ X = \bigcup_{n \in \NN} \{(x, y) \in \RR^2 | (x- \frac{1}{n})^2 +y ^2 = \frac{1}{n^2}\} $ be the Hawaiian Earring space. Brazas \cite[Example 3.8]{B1} introduced a connected semicovering $p :\tx \to X $ with discrete fibers which is not a covering map. Put $\hat{X} = p^{-1}(X \setminus  ((0, 1]) \times \{0\}) $, then $\hat{X}$ is path connected. It is easy to see that every loop in $ \hat{X}$ is null homotopic. Also, $q = p|_{\hat{X}} :\hat{X} \to X$ is a local homeomorphism with $q_* (\pi_1 (\hat{X}, \hat{x}_0))= \{ 1\} \leq \pi_1(X)$.
Calcut and McCarthy \cite[Theorem 1]{Calcut} proved that for a connected and locally path connected space $X$, semilocally simply connectedness of $X$ is equivalent to openness of the trivial subgroup in $\pi_{1}^{qtop}(X)$. Hence the trivial subgroup is not open in $\pi_{1}^{qtop}(\he)$ since $\he$ is not semilocally simply connected at the point $(0,0)$. This implies that $q_* (\pi_1 (\hat{X}, \hat{x}_0))$ is not open in $\pi_{1}^{qtop}(\he)$. Since $q$ can be extended to the semicovering map $p$, $q$ is a subsemicovering map. Note that $q$ is not a full subsemicovering map since otherwise there exists a semicovering map $r: \ty \to \he$ such that $r_*(\pi_1(\ty, \tilde{y}_0)) =  q_*(\pi_1(\hat{X}, \hat{x}_0))$. By Corollary $\ref{c semi}$ $r_*(\pi_1(\ty, \tilde{y}_0))$ is open in $\pi_{1}^{qtop}(\he)$ but $q_* (\pi_1 (\hat{X}, \hat{x}_0))$ is not open in $\pi_{1}^{qtop}(\he)$ which is a contradiction.
\end{ex}

In the following, we define a strong version of the unique path lifting property in order to find a necessary condition for a local homeomorphism to be subsemicovering.
\begin{defn} \label{strong}
Let $p:\tx \to X $ be a local homeomorphism and $f: [0,\alpha) \to X$ be an arbitrary continuous map, $\tilde{f}:[0,\alpha) \to \tx$ be the incomplete lifting of $f$ defined in Lemma \ref{cts} with starting point $\tilde{x}_0\in p^{-1}(f(0))$. Then, we say that $p$ has the {\it strong unique path lifting property} (strong UPLP for short) if there exist $ \varepsilon_{(f,\tilde{x}_0)} >0$ and an open set $U_{(f,\tilde{x}_0)} \subseteq \tx $ such that $\tilde{f}(\alpha-\varepsilon_{(f,\tilde{x}_0)},\alpha)\subseteq U_{(f,\tilde{x}_0)}$ and $p|_{U_{(f,\tilde{x}_0)}}:U_{(f,\tilde{x}_0)} \to p(U_{(f,\tilde{x}_0)})$ is one-to-one. Note that $p|_{U_{(f,\tilde{x}_0)}}$ is a homeomorphism since it is open. We call ${U_{(f,\tilde{x}_0)}}$ a strong neighborhood.
\end{defn}

In the following lemma, we show that every local homeomorphism with strong UPLP has UPLP .
\begin{lem}\label{SUPLP}
If a local homeomorphism has strong UPLP, then so has UPLP.
\end{lem}
\begin{proof}
Suppose that $p:\tx \to X $ is a local homeomorphism with strong UPLP  but it does not have UPLP. Then there exist $ \tilde{f}_1$ and  $ \tilde{f}_2$ such that $p \circ \tilde{f}_1 = f = p \circ \tilde{f}_2$ and $\tilde{f}_1(0) = \tilde{f}_2(0) = \tilde{x}_0 $. Put $A = \{ t \in I | \tilde{f}_1(t) =\tilde{f}_2(t)\}$.  We show that $A$ is an open subset of $I$. Let $a \in A$, then $ \tilde{f}_1(a)= \tilde{f}_2(a)= b $. Since $(\tx, p)$ is a local homeomorphism, there exists $V \subseteq \tx $ such that $ b \in V , \  p|_V:V \to p(V) $ is a homeomorphism. Clearly $ W = \tilde{f}_1^{-1}(V)\cap \tilde{f}_2^{-1}(V)$ is an open subset of $I$. Let $ w \in W$, then $p \circ \tilde{f}_1(w)= p \circ \tilde{f}_2(w) $ and $\tilde{f}_1(w), \tilde{f}_2(w) \in V  $ and $p|_V$ is one-to-one, and so $\tilde{f}_1(w)= \tilde{f}_2(w)$ hence $ W \subseteq A$. Thus $A$ is an open subset of $I$. Since $0 \in A$, we can consider the connected component containing $0$ in $A$, $C$ say, which is open and connected so there exists $\alpha \in I$ such that $C =[0, \alpha)$. Define $\lambda_n : B_n = [1- \frac{1-\alpha}{2 \times n}, 1- \frac{1-\alpha}{2 \times (n+1)}] \to [0, 1] $ by $\lambda_n (t) = (\frac{t - ( 1- \frac{1-\alpha}{2 \times n})}{ (1- \frac{1-\alpha}{2 \times (n+1)})- (1- \frac{1-\alpha}{2 \times n})}) $ for all $n \in \NN$. Note that $\lambda_n$ is a homeomorphism. Now, we define
$$\tilde{h}(t)=\begin{cases} \tilde{f}_1(t) & t \in \ [0, 1 - \frac{1-\alpha}{2 \times 1})
\\ \tilde{f}_1((\lambda_n (t) \times (\alpha- \frac{\alpha}{n+1 }))+((1- \lambda_n (t))\times (1 - \frac{1-\alpha}{n+1}))) &   t \in B_n, n = 4i+1, i \in \NN_0
 \\ \tilde{f}_2((\lambda_n (t) \times (\alpha+\frac{1-\alpha}{n+2}) )+((1- \lambda_n (t))\times (\alpha- \frac{\alpha}{n+1 }))) &   t \in B_n, n = 4i+2, i \in \NN_0
 \\ \tilde{f}_2((\lambda_n(t) \times (\alpha- \frac{\alpha}{n+2 }))+((1- \lambda_n (t))\times (\alpha+\frac{1-\alpha}{n+2}))) &  t \in B_n, n = 4i+3, i \in \NN_0
 \\ \tilde{f}_1((\lambda_n (t) \times (\alpha+\frac{1-\alpha}{n+3}) )+((1- \lambda_n (t))\times (\alpha- \frac{\alpha}{n+2}))) &  t \in B_n, n = 4i, i \in \NN
\end{cases} $$
and put $h= p \circ \tilde{h}$. It is easy to see that $h, \tilde{h} $ are continuous map. Since $p$ has strong UPLP, there exist $ \varepsilon_{(h,\tilde{x}_0)} >0$, an open set $U_{(h,\tilde{x}_0)} \subseteq \tx $ such that $\tilde{h}(1-\varepsilon_{(h,\tilde{x}_0)},1)\subseteq U_{(h,\tilde{x}_0)}$ and $p|_{U_{(h,\tilde{x}_0)}}:U_{(h,\tilde{x}_0)} \to p(U_{(h,\tilde{x}_0)})$ is one-to-one. There exists $n \in \NN$ such that $\frac{1-\alpha}{2 \times n} < \varepsilon_{(h,\tilde{x}_0)} $ and by the definition of $\tilde{h}$ we have $ \tilde{f}_1(\alpha) \in  \tilde{h}(1-\varepsilon_{(h,\tilde{x}_0)},1)$ and $ \tilde{f}_2(\alpha) \in  \tilde{h}(1-\varepsilon_{(h,\tilde{x}_0)},1)$. But $ \tilde{f}_1(\alpha) \neq \tilde{f}_2(\alpha) $ and $p \circ \tilde{f}_1(\alpha) = p \circ \tilde{f}_2(\alpha)$ which contradicts to the injectivity of $p|_{U_{(h,\tilde{x}_0)}}$.
\end {proof}

There exists a local homeomorphism with UPLP which does not have strong UPLP.
\begin{ex}\label{st ex}
\begin{figure}[h!]
\centering
\includegraphics[scale=.5]{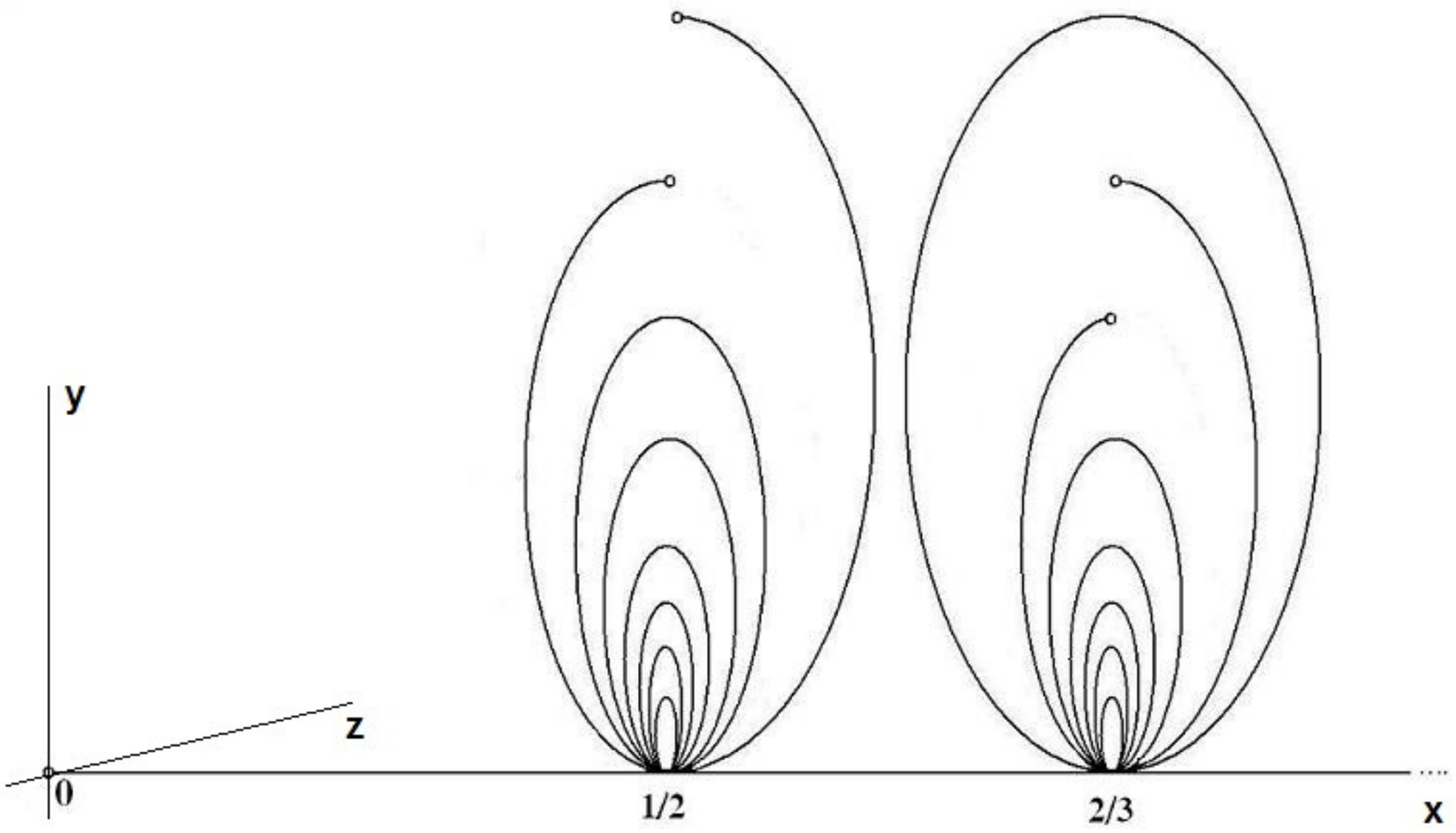}
\caption{$\tilde{X}$}\label{fig1}
\end{figure}

Let $X=\he = \bigcup_{n \in \NN}\{(x, y) \in \RR^2 | (x- \frac{1}{n})^2 +y ^2 = \frac{1}{n^2}\}$ be the Hawaiian Earring space. Put $\it{W_i} = \bigcup_{n \in \{\NN \setminus \{i , i+1\} \}}\{(y, z) \in \RR^2 | (y- \frac{1}{n})^2 +z ^2 = \frac{1}{n^2}\}$ and
$$ S_i=\bigcup \{(y, z)| (y - (1- \frac{1}{i}))^2 + z^2 =  (\frac{1}{i})^2 , z > 0 \}
\bigcup  \{(y, z)| (y - (1- \frac{1}{i+1}))^2 + z^2 =  (\frac{1}{(i+1)})^2 , z < 0 \} $$ for every $i \in \NN$. Let  $\tx = ((0, 1) \times \{ 0\} \times \{ 0\}) \bigcup_{i = 1}^{\infty}( \{1- \frac{1}{i+1}\} \times  (\it{W_i} \cup S_i))$  be a subset of $ \RR^3 $ (see Figure 1). We define $p:\tx \to X $ by
$$p(x, y, z) = \begin{cases} (y, z) & x = 1- \frac{1}{i+1},\ i \in\NN\\
\\
\frac{1}{i}(1 + \cos( \frac{2 \pi}{1-x}), \sin( \frac{2 \pi}{1-x}))  &  1- \frac{1}{i} < x < 1- \frac{1}{i+1},\ i\in\NN. \\
 \end{cases}$$
 It is routine to check that $p$ is a local homeomorphism which has UPLP. Let $\alpha : I \to X$ be a loop defined by $$\alpha (t)= \begin{cases} (0, 0) & t \in [0, \frac{1}{2}] \cup \{1\}
\\ \frac{1}{i}(1 + \cos( \frac{2 \pi}{1-t}), \sin( \frac{2 \pi}{1-t}))  &  1- \frac{1}{i} \leq t \leq 1- \frac{1}{i+1},\ i\in\NN \setminus \{1\}. \\
 \end{cases}$$
 The loop $\alpha$ has no lifting with starting point $(\frac{1}{2}, 0, 0)$ and the incomplete lifting of $\alpha$ with starting point $(\frac{1}{2}, 0, 0)$ is $\tilde{\alpha}: [0,1) \to \tx$ defined by $$\tilde{\alpha}(t)= \begin{cases} (\frac{1}{2}, 0, 0) & t \in [0, \frac{1}{2}]
\\ (t, 0, 0)  &  t \in [\frac{1}{2}, 1).
 \end{cases}$$
 Thus $\tilde{\alpha}$ does not have any strong neighborhood. Therefore $p$ does not have strong UPLP.
\end{ex}

In the following theorem, we show that the strong UPLP is a necessary condition for a local homeomorphism to be a subsemicovering map.
\begin{thm} \label{strong UPLP}
If $p$ is a subsemicovering map, then $p$ has strong UPLP.
\end{thm}
\begin{proof}
If $p:\tx \to X $ is a semicovering map, then it is easy to check that $p$ has strong UPLP. Suppose $p$ is subsemicovering which is not a semicovering map. So there exists a semicovering map $ q: \tilde{Y} \to X $ with an embedding map $\varphi: \tx \to  \tilde{Y}$ such that $q \circ \varphi = p $.
Since $p$ is not semicovering, there exists a path $f$ in $X$ with no lifting. By Lemma \ref{cts}, there exists $\tilde{f}:[0, \alpha) \to \tx$ with starting point $\tilde{x}_0\in p^{-1}(f(0))$ such that $ p \circ \tilde{f} = f $. Also, since $q$ is a semicovering map, $q$ has PLP. Thus there exists a lifting $\hat{f}$ of $f$ in $\tilde{Y}$ with starting point $\varphi(\tilde{x}_0 )$ and $ \varphi(\tilde{f}([0, \alpha))) = \hat{f}|_{[0, \alpha)} $. Since $q$ is a semicovering map, there exists an open neighborhood $U$ at $\hat{f}(\alpha)$ such that $p|_U : U \to p(U)$ is a homeomorphism. Put $U_{(f,\tilde{x})} = \varphi^{-1}(U) \cap \tx $, then there exists $\epsilon > 0$ such that $\tilde{f}(\alpha- \epsilon, \alpha)\subseteq U_{(f,\tilde{x})}$. Also, $p:U_{(f,\tilde{x})} \to p(U_{(f,\tilde{x})})$ is one-to-one since $ q: U \to q(U)$ is a homeomorphism.
\end {proof}

Steinberg \cite[Theorem 4.6]{sting} proved that the condition {\it ``if $f$ is a path in $\tx$ with $p \circ f$ null homotopic (in $X$), then $f(0) = f(1)$''} is a necessary condition for a local homeomorphism $p:\tx \to X $ to be subcovering. In the following theorem, we show that this condition is also a necessary condition for a local homeomorphism to be subsemicovering.
\begin{thm} \label{lazem}
If $p :(\tx, \tilde{x}_0) \to (X, x_0) $ is a subsemicovering map, then
\begin{enumerate}
\item
$p :(\tx, \tilde{x}_0) \to (X, x_0)$ is a local homeomorphism;
\item
if $f$ is a path in $\tx$ with $p \circ f$ null homotopic (in $X$), then $f(0) = f(1)$. $(\bigstar) $
\end{enumerate}
\end{thm}
\begin{proof}
Let $ p' :(\tilde{Y}, \tilde{y}_0) \to (X, x_0)$ be a semicovering map which is an extension of $p$ via an embedding $\varphi : \tx \to \tilde{Y}$, i.e, $p' \circ \varphi =p$. Consider $\tilde{x}$ to be an arbitrary element of $\tx$. Since $p'$ is a local homeomorphism, there exists an open neighborhood $W$ of $ \varphi(\tilde{x})$ such that $p'|_W: W \to p'(W)$ is a homeomorphism. Since $\varphi$ is an embedding, $\varphi^{-1}(W)$ is an open neighborhood of $\tilde{x}$ and $p|_{\varphi^{-1}(W)}: \varphi^{-1}(W) \to p(\varphi^{-1}(W))$ is a homeomorphism. Hence $p :(\tx, \tilde{x}_0) \to (X, x_0)$ is a local homeomorphism.
If $f$ is a path in $\tx$ and $p \circ f$ is null homotopic, then by the definition of a semicovering map, there exists $\tilde{f}: I \to \tilde{Y} $ with starting point $\varphi( f (0))$ such that $p' \circ \tilde{f} = p \circ f $. By Corollary \ref{cor homotopy}, $\tilde{f}$ is null homotopic in $ \tilde{Y}$ since $\tilde{f}$ is a lifting of $p \circ f$. Thus $\tilde{f}$ is a loop. Also $\varphi (f)= \tilde{f}$ since $\tilde{f}$ and $\varphi (f)$ are two liftings of $p \circ f$ with starting point $\varphi( f (0))$. Since $\varphi$ is an embedding and $\varphi (f)= \tilde{f}$, $f$ is a loop.
\end{proof}

In the following, we are going to find a sufficient condition for extending a local homeomorphism to a semicovering map. For this purpose first, note that Steinberg in \cite[Theorem 4.6]{sting} presented a necessary and sufficient condition for a local homeomorphism $p: \tx \to X $ to be subcovering.  More precisely, he proved that a continuous map $p :\tx \to X $ of locally path connected and semilocally simply connected spaces is subcovering if and only if $p:\tx \to X$ is a local homeomorphism and any path $f$ in $\tx$ with $p \circ f$ null homotopic (in $X$) is closed, i.e, $f(0) = f(1)$.
We will show that the latter condition on a local homeomorphism $p :(\tx, \tilde{x}_0) \to (X, x_0) $ is a sufficient condition for $p$ to be subsemicovering provided that $p_* (\pi_1 (\tx, \tilde{x}_0))$ is an open subgroup of the quasitopological fundamental group $\pi_1^{qtop}(X, x_0)$.

\begin{thm} \label{asli}
Let $p :(\tx, \tilde{x}_0) \to (X, x_0) $ be a map such that $p_* (\pi_1 (\tx, \tilde{x}_0))$ is an open subgroup of $\pi_1^{qtop}(X, x_0)$. Then $p$ is a subsemicovering map if and only if
\begin{enumerate}
\item
$p :(\tx, \tilde{x}_0) \to (X, x_0)$ is a local homeomorphism;
\item
if $f$ is a path in $\tx$ with $p \circ f$ null homotopic (in $X$), then $f(0) = f(1)$.
\end{enumerate}
\end{thm}
\begin{proof}
 The necessity follows by Theorem $\ref{lazem}$. For sufficiency, using Theorem $\ref{semi}$, let $ p' :(\tilde{Y}, \tilde{y}_0) \to (X, x_0)$ be the semicovering map associated to the open subgroup $p_* (\pi_1 (\tx, \tilde{x}_0))$. Since semicoverings have lifting criterion (see Theorem \ref{criterion}), by lifting $p$ to $(\tilde{Y}, \tilde{y}_0)$, we obtain a mapping $\varphi : (\tx, \tilde{x}_0) \to (\tilde{Y}, \tilde{y}_0)$ such that $p'\circ\varphi=p$. First, we show that $\varphi$ is injective. Suppose $\varphi(\tilde{x}_1) = \varphi(\tilde{x}_2)$. Let $f_j : (I, 0, 1) \to (\tx, \tilde{x}_0, \tilde{x}_j ), \ j = 1, 2,$ be two paths. Note that we use here notation $[f_j], \ j = 1, 2$ for the homotopy classes in the fundamental groupoid (see \cite[Section 1.7]{s}). Then $\varphi([f_2])\varphi([f_1^{-1}]) \in \pi_{1}(\tilde{Y}, \tilde{y}_0)$ so $p'(\varphi([f_2])\varphi([f_1^{-1}] ) \in  p'_* (\pi_{1}(\tilde{Y}, \tilde{y}_0))$ and $ p'(\varphi[f_2])\varphi([f_1^{-1}] ) =  p([ f_2])p([ f_1^{-1}]) $ since $p' \circ \varphi = p$. Note that $p_* (\pi_1 (\tx, \tilde{x}_0))=  p'_* (\pi_{1}(\tilde{Y}, \tilde{y}_0))$ so there is a loop $f$ at $x_0$ such that
$$p([f]) = p([f_2])p([f_1^{-1}]).$$
Therefore
$$p([f_2^{-1}(ff_1)])=p([f_2^{-1}]) p([f_2])p([f_1^{-1}])p([f_1]) =[1_{p(\tilde{x}_2)}].$$
Hence, by assumption, $f_2^{-1}(ff_1)$ is a loop (at $\tilde{x}_2$), whence $\tilde{x}_1 = \tilde{x}_2$.
It remains to show $\varphi : (\tx, \tilde{x}_0) \to (\tilde{Y}, \tilde{y}_0)$ is an open map. To prove this, we show that $\varphi$ is a local homeomorphism since every local homeomorphism is an open map. It is enough to show that for an arbitrary element $\tilde{x}$ of $\tx$, there exists an open neighborhood $W_{ \tilde{x}}$ such that $\varphi|_{W_{ \tilde{x}}}: W_{ \tilde{x}} \to \varphi(W_{ \tilde{x}})$ is a homeomorphism. If $V_{ \tilde{x}}$ is an open neighborhood obtained from local homeomorphism $p$ at $ \tilde{x}$ and $U_{\varphi (\tilde{x})}$ is an open neighborhood obtained from local homeomorphism $p'$ at $\varphi (\tilde{x})$, then $W_{ \tilde{x}} = V_{ \tilde{x}}\bigcap \varphi^{-1}(U_{ \varphi (\tilde{x})})$. Since $ \varphi$ is continuous, $ W_{ \tilde{x}}$ is open in $ \tx$. Also $\varphi|_{W_{ \tilde{x}}} = p|_{W_{ \tilde{x}}}\circ {p'^{-1}}|_{U_{\varphi (\tilde{x})}}$ hence $\varphi|_{W_{ \tilde{x}}}$ is a homeomorphism. Thus $\varphi$ is a local homeomorphism.
\end{proof}

The following corollary is a consequence of Theorems \ref{asli} and \ref{semi}.
\begin{cor} \label{C asli}
A map $p :(\tx, \tilde{x}_0) \to (X, x_0) $ is a full subsemicovering map if and only if
\begin{enumerate}
\item
$p :(\tx, \tilde{x}_0) \to (X, x_0)$ is a local homeomorphism;
\item
if $f$ is a path in $\tx$ with $p \circ f$ null homotopic (in $X$), then $f(0) = f(1)$;
\item
$p_* (\pi_1 (\tx, \tilde{x}_0))$ is an open subgroup of $\pi_1^{qtop}(X, x_0)$.
\end{enumerate}
\end{cor}
\begin{proof}
 Since every full subsemicovering map is a subsemicovering map, the necessity of conditions $(1)$ and $(2)$ are obtained by Theorem $\ref{lazem}$. To prove condition $(3)$, let $p$ can be extended to a semicovering map $q: (\ty, \tilde{y}_0) \to (X, x_0)$ such that $ p_*(\pi_1(\tilde{X}, \tilde{x}_0)) =  q_*(\pi_1(\tilde{Y}, \tilde{y}_0))$. Hence $p_*(\pi_1(\tilde{X}, \tilde{x}_0))$ is open in $\pi_1^{qtop}(X, x_0)$ since $q_*(\pi_1(\tilde{Y}, \tilde{y}_0))$ is open in $\pi_1^{qtop}(X, x_0)$ by Corollary $\ref{c semi}$. Sufficiency is obtained similar to the proof of Theorem $\ref{asli}$.
\end{proof}
The following corollary can be concluded by the classification of connected covering spaces of $X$, Theorem $\ref{cov}$, and Theorem \ref{asli}.

\begin{cor} \label{C asli}
A map $p :(\tx, \tilde{x}_0) \to (X, x_0) $ is a full subcovering map if and only if
\begin{enumerate}
\item
$p :(\tx, \tilde{x}_0) \to (X, x_0)$ is a local homeomorphism;
\item
if $f$ is a path in $\tx$ with $p \circ f$ null homotopic (in $X$), then $f(0) = f(1)$;
\item
$p_* (\pi_1 (\tx, \tilde{x}_0))$ contains an open normal subgroup of $\pi_1^{qtop}(X, x_0)$.
\end{enumerate}
\end{cor}

We need the following lemma for the next example.
\begin{lem}\label{hamid}
Let $p:\tx \to X$ be a local homeomorphism. Suppose that $\tx$ is Hausdorff and every null homotopic loop $\alpha$ in $X$ is of the form $\Pi_{i=1}^{n} \alpha_i$, where $$\alpha_i (t)= \begin{cases} (f_i \circ \lambda_i)(t) & t \in [0, a_i]
\\ (f^{-1}_i \circ \gamma_i)(t)  &  t \in [a_i, 1], \\
 \end{cases}$$
in which $0\leq a_i\leq 1$, $f_i$ is a path in $X$, $\lambda_i: [0, a_i] \to [0, 1]$ is defined by $\lambda_i(t) = \frac{t}{a_i}$ and $\gamma_i: [a_i, 1] \to [0, 1]$ is defined by $\gamma_i(t) = \frac{t- a_i}{1- a_i}$,  for every $i \in \NN$. Then $p$ has the condition $(\bigstar)$ in Theorem \ref{lazem}.
\end{lem}
\begin{proof}
Let $\tilde{\alpha}$ be a path in $\tx$ such that $p\circ\tilde{\alpha} = \alpha $ is null homotopic in $X$. By the hypothesis, without loss of generality we can assume that
 $$\alpha (t)= \begin{cases} (f \circ \lambda)(t) & t \in [0, a]
\\ (f^{-1} \circ \gamma)(t)  &  t \in [a, 1], \\
 \end{cases}$$
  where $f$ is a path in $X$, $\lambda: [0, a] \to [0, 1]$ defined by $\lambda(t) = \frac{t}{a}$, $\gamma: [a, 1] \to [0, 1]$ defined by $\gamma(t) = \frac{t- a}{1- a}$ and $a \in [0, 1]$. Put $A = \{t \in [0, a]| \tilde{\alpha}(t)= (\tilde{\alpha} \circ \eta)(t) \}$ where $\eta: [0,a] \to [a, 1]$ is defined by $\eta(t) = t + 1 - \frac{t}{a}$. We show that $A$ is a nonempty clopen subset of $[0, a]$ which implies that $ A = [0, a]$. Clearly $A$ is nonempty since for $t =a $ we have $ \tilde{\alpha}(a) = (\tilde{\alpha} \circ \eta)(a)$. Let $b \in A$, then $ \tilde{\alpha}(b)=(\tilde{\alpha} \circ \eta)(b)=c $. Since $p$ is a local homeomorphism, there exists $V \subseteq \tx $ such that $ c \in V$ and $p|_V:V \to p(V) $ is a homeomorphism. Put $ W = (\tilde{\alpha}^{-1}(V)\cap (\tilde{\alpha} \circ \eta)^{-1}(V))\cap [0, a]$, then $W$ is an open subset of $[0, a]$. If $ w \in W$, then $(p \circ \tilde{\alpha})(w) = \alpha (w) = (f \circ \lambda)(w) = f (\frac{w}{a}) $. Also since $(w + 1 - \frac{w}{a}) \in [a, 1]$, we have $(p \circ (\tilde{\alpha} \circ \eta))(w) = \alpha (\eta(w)) = \alpha (w + 1 - \frac{w}{a}) = (f^{-1} \circ \gamma)(w + 1 - \frac{w}{a}) = f^{-1} (\frac{(w + 1 - \frac{w}{a}) - a}{1- a})=f (\frac{w}{a}) $. Hence $(p \circ \tilde{\alpha})(w)= f (\frac{w}{a})= (p \circ (\tilde{\alpha} \circ \eta))(w)$. Since $p$ is one to one on $V$ and $\tilde{\alpha}(w), (\tilde{\alpha} \circ \eta)(w) \in V  $, we have $\tilde{\alpha}(w)= (\tilde{\alpha} \circ \eta)(w)$. Hence $ W \subseteq A$ and therefore $A$ is open. Now we show that $ A$ is closed. Let $ b \in [0, a] \setminus A$ then $\tilde{\alpha}(b) \neq (\tilde{\alpha} \circ \eta)(b)$. Since $\tx$ is Hausdorff and $p$ is a local homeomorphism, there exist open neighborhoods $V_{\tilde{\alpha}(b)} $ of $\tilde{\alpha}(b)$ and $V_{(\tilde{\alpha} \circ \eta)(b)} $ of $(\tilde{\alpha} \circ \eta)(b)$ such that $ V_{\tilde{\alpha}(b)} \cap V_{(\tilde{\alpha} \circ \eta)(b)} = \phi$ and $p|_{V_{\tilde{\alpha}(b)}} : V_{\tilde{\alpha}(b)} \to p (V_{\tilde{\alpha}(b)}), \ p|_{V_{(\tilde{\alpha} \circ \eta)(b)}} : V_{(\tilde{\alpha} \circ \eta)(b)} \to p (V_{(\tilde{\alpha} \circ \eta)(b)})$ are homeomorphisms. Put $W = (\tilde{\alpha}^{-1}(V_{\tilde{\alpha}(b)}) \cap g^{-1}(V_{g(b)}))\cap [0, a]$, then $W$ is an open subset of $[0, a]$ and $ b \in W$ and $W \subseteq [0, a] \setminus A$ which implies that $A$ is closed.
  Hence $0\in A $ which implies that $\tilde{\alpha}(0)= \tilde{\alpha} \circ \eta (0)$. Thus $\tilde{\alpha} (0) = \tilde{\alpha}(1)$ and so $p$ has the condition $(\bigstar)$.
\end {proof}

The following example shows that the condition $(\bigstar)$ is not a sufficient condition for $p$ to be subsemicovering. Hence we can not omit openness of $p_* (\pi_1 (\tx, \tilde{x}_0))$  from the hypotheses of Theorem \ref{asli}.
\begin{ex}\label{ex222}
Let $p: \tx \to X =\he$ be the local homeomorphism introduced in Example \ref{st ex}. By using Lemma \ref{hamid}, $p$ has the condition $(\bigstar)$ . Let $\alpha : I \to X$ be a loop introduced in Example \ref{st ex}. The path $\alpha$ has no lifting with starting point $(\frac{1}{2}, 0, 0)$ and the incomplete lifting of $\alpha$ with starting point $(\frac{1}{2}, 0, 0)$ is $\tilde{\alpha}: [0,1) \to \tx$, introduced in Example \ref{st ex}. $\tilde{\alpha}$ does not have any strong neighborhood. Therefore $p$ does not have strong UPLP and so it is not subsemicovering (see Theorem \ref{strong UPLP}).
\end{ex}

If $p :(\tx, \tilde{x}_0) \to (X, x_0)$ is a local homeomorphism, then openness of $p_* (\pi_1 (\tx, \tilde{x}_0))$ is not a necessary condition for $p$ to be a subsemicovering map. The following example gives a subsemicovering map $p :\tx  \to \he$ in which $p_* (\pi_1 (\tx, \tilde{x}_0))$ is not an open subgroup of $\pi_{1}^{qtop}(\he)$.
\begin{ex}\label{ex1}
Let $q:\hat{X} \to X $ be the local homeomorphism introduced in Example \ref{ex111}. We recall that $X = \he$ and $q_* (\pi_1 (\hat{X}, \hat{x}_0))= \{ 1\} \leq \pi_1(\he)$. It is known that $\{ 1\}$ is not open in $\pi_{1}^{qtop}(\he)$ since $\he$ is not semilocally simply connected (see Example \ref{ex111}). Hence $q_* (\pi_1 (\hat{X}, \hat{x}_0))$ is not open in $\pi_{1}^{qtop}(\he)$ but $q$ is a subsemicovering map.
\end{ex}

If $p :(\tx, \tilde{x}_0) \to (X, x_0)$ is a local homeomorphism with condition $(\bigstar)$ and $p_* (\pi_1 (\tx, \tilde{x}_0))= \{1\}$, then $p$ is not necessarily a subsemicovering map. See the following example.
\begin{ex}\label{ex200}
Let $p: \tx \to X = \he$ be the local homeomorphism introduced in Example \ref{st ex}. Put $\hat{X} = \tx \setminus \{(r, s, 0) \in \RR^3 | r \in  \{(1- \frac{1}{i+1})| i \in \NN\}, \  s \in (0, 1]   \}$, then $\hat{X}$ is path connected. It is easy to see that every loop in $ \hat{X}$ is null homotopic. Also, $q = p|_{\hat{X}} :\hat{X} \to X$ is a local homeomorphism with $q_* (\pi_1 (\hat{X}, \hat{x}_0))= \{ 1\} \leq \pi_1(X)$. By using Lemma \ref{hamid}, $q$ has the condition $(\bigstar)$ . Let $\alpha : I \to X$ be a loop defined by $$\alpha (t)= \begin{cases} (0, 0) & t \in [0, \frac{1}{2}] \cup \{1\}
\\ \frac{1}{i}(1 + \cos( \frac{2 \pi}{1-t}), \sin( \frac{2 \pi}{1-t}))  &  1- \frac{1}{i} \leq t \leq 1- \frac{1}{i+1},\ i\in\NN \setminus \{1\}. \\
 \end{cases}$$
 The loop $\alpha$ has no lifting with starting point $(\frac{1}{2}, 0, 0)$ and the incomplete lifting of $\alpha$ with starting point $(\frac{1}{2}, 0, 0)$ is $\hat{\alpha}: [0,1) \to \hat{X}$ defined by $$\hat{\alpha}(t)= \begin{cases} (\frac{1}{2}, 0, 0) & t \in [0, \frac{1}{2}]
\\ (t, 0, 0)  &  t \in [\frac{1}{2}, 1).
 \end{cases}$$ Thus $\hat{\alpha}$ does not have any strong neighborhood. Therefore $q$ does not have strong UPLP. Since strong UPLP is a necessary condition for $q$ to be a subsemicovering map, $q$ is not a subsemicovering map.

\end{ex}

By Theorems \ref{strong UPLP} and \ref{lazem}, the strong UPLP and the condition $(\bigstar)$ are two necessary conditions for a local homeomorphism to be a subsemicovering map.
It is natural to ask the relationship between these two necessary conditions.
The following example shows that the strong UPLP does not imply the condition $(\bigstar)$, even if $p_* (\pi_1 (\tx, \tilde{x}_0))$ is an open subgroup of $\pi_1^{qtop}(X, x_0)$.

\begin{ex}\label{ex2}
Let $ X = D^2 $ be the disk in $\RR^2$ and $\tx = \{(x,y)| 0 \leq y \leq \frac{1}{2} , x \in \RR\}$. Define $p:\tx \to X $ by
$$p(x,y) = \frac{1}{1+y} e^{2\pi x}.$$
It is routine to check that $p$ is a local homeomorphism with strong UPLP which does not have the condition $(\bigstar)$. Note that $p$ is not one-to-one and since $D^2$ is simply connected, $p_* (\pi_1 (\tx, \tilde{x}_0))$ is an open subgroup of $\pi_1^{qtop}(D^2) = \{1\}$.
\end{ex}
In the above example, $X$ is simply connected and $\tx$ is path connected but $p$ is not one-to-one since $p$ is a map without the condition $(\bigstar)$.
More precisely, if $p:(\tx, \tilde{x}_0) \to (X, x_0)$ is a map with the condition $(\bigstar)$ and $X$ is simply connected and $Y$ is path connected space, then $p$ is one-to-one.

The following example shows that the condition $(\bigstar)$ does not imply the strong UPLP.
\begin{ex}
Let $p: \tx \to X$ be the local homeomorphism introduced in Example \ref{st ex}. It is easy to check that $p$ is a local homeomorphism with the condition $(\bigstar)$ (see Example \ref{ex222}) and we recall that $p$ does not have strong UPLP (see Example \ref{st ex}).
\end{ex}

By extending the notions strong homotopy and the fundamental inverse category and monoid introduced by Steinberg \cite[Section 3]{sting} to semicovering maps, we give a necessary and sufficient condition for a subsemicovering map to be semicovering. Note that, the same necessary and sufficient condition for a subcovering map to be covering is not stated in \cite{sting}.
First, we recall the notion strong homotopy equivalence. Let $f: I \to X$  be a path and $f_t : I \to X$ given by $f_t(s) = f(ts )$, for every $t \in I$. It is convenient to think of $f_t$ as the prefix of $f$ of length $t$. Use the notation $\sim _h$ for the equivalence relation of being homotopic relative to base points. For two paths $f, g: (I, 0, 1) \to (X, x_0, y_0)$, $f$ is strongly homotopic to $g$, denoted by $f \sim _s g$, if  $f \sim _h g$ and for every $t \in I$ there exists $t' \in I$ such that $f_t \sim _h g_t'$, and vice versa.
As an example, any two reparametrizations of the same path are strongly homotopic.
If $f$ is a path in $X$, then we use $[f]_s$ to denote its strong homotopy class and $[f]$ to denote its homotopy class. Also, Steinberg defined a category $\mu_1(X)$ with involution. He showed that $\mu_1(X)$ is an inverse category (see \cite[Proposition 3.2]{sting}) and called $\mu_1(X)$ the fundamental inverse category of $X$. If $x \in X$, then the local inverse monoid at $x$ is denoted by $\mu_1(X, x) $ and it is called the fundamental inverse monoid of $X$ at $x$.

Steinberg \cite[Section 4]{sting} obtained the following results (Lemma \ref{kosh}, Lemma \ref{Path lifting} and Theorem \ref{lifting Subsemicovering}) for subcoverings of a semilocally simply connected space (see \cite[Lemma 4.1]{sting}, \cite[Lemma 4.7]{sting} and \cite[Theorem 4.8]{sting}).
Similarly, according to lifting criterion and homotopy lifting property for semicovering maps, we can state and prove the following results for an arbitrary subsemicovering map.

\begin{lem}\label{kosh} (\cite[Lemma 4.1]{sting}).
Let $p:(Y, y_0) \to (X, x_0)$ be a semicovering map and suppose $f_1, f_2:(I, 0, 1) \to (Y, y_0 , y_1)$ are paths such that $p([f_1]_s) \leq p([f_2]_s) $. Then $[f_1]_s \leq [f_2]_s$. In particular, if $p \circ f_1 \sim_s p \circ f_2,$ then $ f_1 \sim_s f_2. $
\end{lem}

\begin{lem}\label{Path lifting} (\cite[Lemma 4.7]{sting})(Path Lifting Property for Subsemicovering).
\\ Let $p :(\tx, \tilde{x}_0) \to (X, x_0)$ be a subsemicovering map and $f: (I, 0) \to (X, x)$ be a path. Then there exists $ \tilde{f}: (I, 0)\to (\tx , \tilde{x}_0)$ such that $p \circ \tilde{f} = f $ if and only if $ [f \bar{f}]_s \in p_* (\mu_1(\tx, \tilde{x}_0))$. Moreover, $\tilde{f}$ is unique when it exists.
\end{lem}

\begin{thm}\label{lifting Subsemicovering} (\cite[Theorem 4.8]{sting})(Lifting Criterion Theorem for Subsemicovering).
\\ Let $p :(\tx, \tilde{x}_0) \to (X, x_0)$ be a subsemicovering map and $g : (Z, z) \to (X, x)$ be a continuous mapping with $Z$ locally path connected. Then there is a lift $ \tilde{g} :(Z, z) \to (\tx, \tilde{x}_0)$ of $g$ if and only if $g_*(\mu_1(Z, z)) \subseteq p_* (\mu_1(\tx, \tilde{x}_0))$. Moreover, $\tilde{g}$ is unique.
\end{thm}

Using Lemmas \ref{kosh}, \ref{Path lifting} and Theorem \ref{lifting Subsemicovering}, we give a necessary and sufficient condition for a subsemicovering map to be semicovering. Note that there is no similar result for subcovering maps in \cite{sting}.
\begin{thm}\label{pp}
Let $p :(\tx, \tilde{x}_0) \to (X, x_0)$ be a subsemicovering map. Then $p$ is a semicovering map if and only if $ [f \bar{f}]_s \in p_* (\mu_1(\tx, \tilde{x}_0))$ for an arbitrary path $f$ in $X$.
\end{thm}
\begin{proof}
Suppose $p$ is a semicovering map and $f: (I, 0) \to (X, x)$ is an arbitrary path in $X$. Since $p$ has PLP, there exists $ \tilde{f}: (I, 0)\to (\tx , \tilde{x}_0)$ such that $p \circ \tilde{f} = f $. Hence $ \tilde{f} \in \mu_1(\tx, \tilde{x}_0)$ and so $ [f \bar{f}]_s \in p_* (\mu_1(\tx, \tilde{x}_0))$.

 Conversely, let $p:\tx \to X $ be a subsemicovering map, then there exists a semicovering map $ q: \tilde{Y} \to X $ with an embedding map $\varphi: \tx \to  \tilde{Y}$ such that $q \circ \varphi = p $. It is enough to show that $p$ has PLP and UPLP (see Theorem \ref{semicovering map}). We can conclude that $p$ has PLP by Theorem \ref{Path lifting}. Since $q$ has UPLP and $p$ can be extended to $q$, the map $p$ has UPLP. Hence $p$ is a semicovering map.
\end {proof}
The following corollary is an immediate consequence of Theorem \ref{pp}.
\begin{cor}\label{c ppp}
Let $p:(\tilde{X}, \tilde{x}_0) \to (X, x_0) $ be a subsemicovering map and $X$ be a space such that every semicover of $X$ is a cover. Then $p$ is a covering map if and only if $ [f \bar{f}]_s \in p_* (\mu_1(\tx, \tilde{x}_0))$ for every path $f: (I, 0) \to (X, x)$.
\end{cor}
If $X$ is a semilocally simply connected space then every semicover of $X$ is a cover (see \cite[Corollary 7.2]{B1}).
\begin{cor}\label{c pppp}
Let $p:(\tilde{X}, \tilde{x}_0) \to (X, x_0) $ be a subcovering map and $X$ be a semilocally simply connected space. Then $p$ is a covering map if and only if $ [f \bar{f}]_s \in p_* (\mu_1(\tx, \tilde{x}_0))$ for every path $f: (I, 0) \to (X, x)$.
\end{cor}
The following corollary is a consequence of Theorem \ref{pp} and \cite[Theorem 4.2]{mt2}.
\begin{cor}\label{c pp}
Suppose $p:(\tilde{X}, \tilde{x}_0) \to (X, x_0) $ is a subsemicovering map and $X$ is locally path connected, such that $[\pi_1(X,x_0): p_*(\pi_1(\tilde{X},\tilde{x}_0))]$ is finite. If $ [f \bar{f}]_s \in p_* (\mu_1(\tx, \tilde{x}_0))$ for every path $f: (I, 0) \to (X, x)$, then $p$ is a covering map.
\end{cor}


\begin{thebibliography}{99}

\bibitem{3}
J. Brazas,  \textit{The topological fundamental group and free topological groups}, Topology and its Applications.
 \textbf{124} (2011), 779802.


\bibitem{B1}
J. Brazas, \textit{ Semicoverings: A generalization of covering space theory},
Homology Homotopy
Appl. \textbf{14} (2012), no. 1, 33-63.



\bibitem{B2}
J. Brazas, \textit{ Semicoverings, coverings, overlays, and open subgroups of the quasitopological fundamental group}, Topology Proceedings Volume 44, (2014), 285-313.

\bibitem{Calcut}
J.S. Calcut, J.D. Mccarthy, \textit{ Discreteness and homogeneity of the topological fundamental group}, Topology Proceedings Volume 34 (2009), 339-349.



\bibitem{mt}
M. Kowkabi, H. Torabi, B. Mashayekhy, \textit{On the category of local homeomorphisms with unique path lifting property}, Proceeding of $24^{th}$ Iranian Algebra Seminar, November 12-13, 2014, 96-99.

\bibitem{mt2}
 M. Kowkabi,B. Mashayekhy, H. Torabi,
 \textit{When is a local homeomorphism a semicovering map?}
, arXiv: 1602.07260v1.

\bibitem{Klevdal}{} C. Klevdal,   \textit{
A Galois Correspondence with Generalized
Covering Spaces}, Undergraduate Honors Theses,(2015), Paper 956.


\bibitem{r}
A. Hatcher, \textit{Algebraic Topology}, Cambridge: Cambridge University Press, (2002).

\bibitem{s}
 E.H. Spanier, {\it Algebraic Topology}, Springer, 3rd Edition (1994).

\bibitem{sting}{}  B. Steinberg,  \textit{ The lifting and classification problems for subspaces
of covering spaces},  Topology and
its Applications.  \textbf{133} (2003),  15-35.


\bibitem{t}
 H. Torabi, A. Pakdaman, B. Mashayekhy:
 \textit{On the Spanier groups
and covering and semicovering map spaces}
, arXiv:1207.4394v1.


\end{thebibliography}
\end{document}